\documentclass{article}

\usepackage{latexsym}
\usepackage[T1]{fontenc}
\usepackage[utf8]{inputenc}
\usepackage[english]{babel}
\usepackage{float}

\usepackage{caption}
\usepackage{xcolor}
\usepackage{mathrsfs}
\usepackage{yfonts}
\usepackage{amssymb}
\usepackage{amsmath}
\usepackage{amsthm}
\usepackage{mathdots}
\usepackage{graphicx}
\usepackage{tikz-cd}
\usepackage{tikz}
\usetikzlibrary{arrows}
\usepackage{amsfonts}
\newtheorem{Thm}{Theorem}
\newtheorem{Def}[Thm]{Definition}

\newtheorem{Lemma}[Thm]{Lemma}
\newtheorem{Cor}[Thm]{Corollary}

\newtheorem{Prop}[Thm]{Proposition}

\newcommand{\numberset}{\mathbb}
\newcommand{\N}{\numberset{N}}

\usepackage{array}

\usepackage{graphicx}
\usepackage{epsfig}
\usepackage{framed}
\usepackage{mathrsfs}
\usepackage{mathtools}

\usepackage [ruled,linesnumbered,noend]{algorithm2e}
\SetAlgoCaptionSeparator{.}
\SetAlgoCaptionLayout{small}
\SetKw{KwGo}{go to}
\SetKw{KwBreak}{break}

\SetCommentSty{mycommfont}

\usepackage{authblk}

\newcommand{\p}{\mathbf{P}}

\title{Randomized algorithms to generate hypergraphs with given degree sequences}
\date{}

\author[1]{Michela Ascolese\thanks{Corresponding author: \texttt{michela.ascolese@unifi.it}}}
\author[2]{Matthias Lienau}
\author[2]{Matthias Schulte}
\author[2]{Anusch Taraz}

\affil[1]{\small{Department of Mathematics and Computer Science, University of Florence}}
\affil[2]{Institute of Mathematics, Hamburg University of Technology}

\begin{document}

\maketitle

\begin{abstract}
    The question whether there exists a hypergraph whose degrees are equal to a given sequence of integers is a well-known reconstruction problem in graph theory, which  is motivated by discrete tomography. In this paper we approach the problem by randomized algorithms which generate the required hypergraph with positive probability if the sequence satisfies certain constraints. 
\end{abstract}

\section{Introduction and Results}
One of the central topics in discrete tomography is the 
reconstruction of a discrete object based on partial knowledge, such as its horizontal and vertical projections, see e.g.~\cite{KubaHerman,KubaHerman2}. 
This task can be rephrased in the context of graph theory as the problem of reconstructing a hypergraph starting from some information concerning its structure, for example about its uniformity and degree sequence. 
In contrast to graphs, this question is NP-hard for hypergraphs. In this paper, we analyse randomized algorithms that find a solution in certain situations.

\paragraph{Notation.}

We briefly introduce some notation that is needed to formulate the basic questions, related results and our contributions. We shall slightly deviate from the standard graph theoretic notions, by first allowing for the possibility that edges contain multiple copies of the same vertex and that the edge set contains multiple copies of identical edges. 
More precisely, a $k$-hypergraph is a pair $H=(V,E)$ where $V=[n]\coloneqq\{1,\dots,n\}$ denotes the set of vertices and $E=\{e_{1},\dots,e_{m}\}$ the multi-set of edges. Here every edge $e_i$ is a multi-set of vertices of cardinality $|e_{i}|=k$. We define $|e|_{i}$ to be the number of occurrences of the vertex $i$ in the edge $e$. The degree of a vertex, $\deg(i)=\sum_{j=1}^{m}|e_{j}|_{i}$, is the number of edges containing $i$ (counted with its multiplicity).

An edge $e$ is called a \emph{loop} if there exists $i\in [n]$ such that $|e|_{i}>1$. 
Two edges $e_{i},e_{j}\in E$ with $e_{i}=e_{j}$ and $i\neq j$ are called \emph{parallel edges}. 
The $k$-hypergraph $H$ is called \emph{simple} if it does not contain loops or parallel edges.

We consider integer sequences
$\pi=(d_{1},\dots,d_{n})$ of length $n$ such that $d_{1}\geq d_{2}\geq\hdots\geq d_{n}\geq 1$, and we define $\sigma=\sum_{i=1}^{n} d_{i}$. 
Such a sequence $\pi$ is called \emph{k-graphic} if there exists a simple $k$-hypergraph $H$ with $d_{i}=\deg(i)$ for all $i$. In this case we say that $H$ has $\pi$ as degree sequence. 

Clearly, for a sequence $\pi$ to be $k$-graphic we must have that $d_1\leq\sigma/k$ and that $k$ divides $\sigma$. Hence we shall always assume from now on that these properties hold for any integer sequence we consider.

The following two tasks are well-known problems in graph theory that are motivated by questions in discrete tomography. Given a number $k\in\N$ and a sequence $\pi=(d_{1},\dots,d_{n})$, 
\begin{itemize}
	\item decide whether $\pi$ is $k$-graphic (decision problem), 
	\item find a simple $k$-hypergraph $H$ that has $\pi$ as degree sequence (reconstruction problem).
\end{itemize}

\paragraph{History.} 
Just as with Satisfiability and Colorability, the borderline between tractability and non-tractability runs between $k=2$ and $k\ge 3$. 
In the case of graphs, both the decision and reconstruction problem can be solved in polynomial time. A non-recursive characterization of graphic degree sequences was given by Erd\H{o}s and Gallai in~\cite{erdos}. Later, many equivalent conditions were provided (see~\cite{7criteria}). Moreover, Hakimi~\cite{hakimi} and Havel~\cite{havel} showed that an intuitive greedy algorithm solves the reconstruction problem in polynomial time. 

Moving to hypergraphs, in 1975 Dewdney~\cite{dewdney} characterized $k$-graphic sequences, but unfortunately his characterization \textcolor{black}{cannot be checked in polynomial time and does not yield a feasible reconstruction algorithm.} The same is true for a characterization given by Billington using the notion of tableaux~\cite{billington}. 

Several papers contributed necessary~\cite{billington,Choudum} or sufficient~\cite{behrens,brlek} conditions. In 2018, Deza et al.~\cite{deza} proved that the decision problem is NP-complete for $k\ge 3$. This hardness result motivated research into subclasses of sequences for which a polynomial time solution can be given. Many of them were identified, and reconstruction algorithms mainly based on greedy techniques were provided (see~\cite{step_seq,maximal_instances,fros,omogenei}).

\paragraph{Randomized approach.}
In this paper we investigate the use of randomized algorithms to generate hypergraphs with a given degree sequence. In combinatorics, the use of randomness to prove the existence of certain structures with prescribed properties is usually called the \textit{probabilistic method} and was pioneered by Erd\H{o}s. Its underlying idea can roughly be described as follows: perform a suitable random experiment, show that with positive probability the outcome yields the desired structure, hence such an object must exist (see e.g.~\cite{AlonSpencer}).

The so-called \emph{configuration model}, initially used for regular graphs, generates random graphs with a given degree sequence, see e.g.\ \cite{Bollobas85} as well as \cite{vdH_book} and the references therein. In this model each vertex is equipped with so-called half-edges, where the number of these half-edges is equal to its desired degree. Then two half-edges are chosen uniformly at random and combined to create an edge until all half-edges are gone. In general, this procedure may yield loops or parallel edges, which are precisely the outcomes that we would like to avoid. Thus, one is interested in the probability that the obtained graph is simple. In \cite{Janson2009}, for example, a necessary and sufficient condition on the degree sequences is given that ensures that this probability does not converge to zero as the numbers of vertices and edges tend to infinity. For a non-asymptotic approximation of the probability we refer to \cite{Angel}.

One can  generalize the configuration model to $k$-hypergraphs directly as done, for example, in \cite{Cooper,CooperFriezeMolloyReed}. Even though the half-edges are not truly half-edges anymore, but rather $1/k$-edges for $k\neq2$, we continue referring to them as half-edges. It seems intuitive that for larger $k$ it becomes more unlikely to draw the exact same edge twice, so the probability to get parallel edges should be small. On the other hand, it becomes more likely to produce loops. The probability that two given half-edges of a vertex $v$ are contained in a same edge is given by $1/(\sigma-1)$. By summing over all choices for pairs of half-edges adjacent to a single vertex and accounting for the fact that a single edge can contain at most $\binom{k}{2}$ pairs of equal vertices, we obtain the lower bound 
\begin{equation}
	\label{eq:configmodel}
\binom{k}{2}^{-1}\sum_{i=1}^n \frac{1}{\sigma-1} \binom{d_i}{2}
\end{equation}for the expected number of loops. Since this expression can tend to infinity with growing $n$, for example when $\min_{i\in[n]}d_i\to\infty$ as $n\to\infty$,  the aim of this paper is to design a model that works better in such scenarios.

\paragraph{Our approach.}
We model the half-edges of the vertices $\{1,\dots,n\}$ as balls that are distributed and then drawn from a suitable number of boxes. Here is the rough idea:

\begin{itemize}
	\item[1.] Consider $k+1$ boxes with labels $1,\dots,k+1$, and for all $i\in[n]$ take $d_i$ balls with label $i$, referring to the vertex $i$. 
	\item[2.] Distribute the balls among the boxes such that all the balls with the same label belong to the same box, and any box contains at most $\sigma/k$ balls.
	\item[3.] Consider the $k$ boxes that contain the highest number of balls and, if there is a tie, take the boxes with the largest labels. Draw one ball from each of these boxes uniformly at random to construct an edge consisting of the labels (i.e.~vertices) of the balls. Repeat until all boxes are empty.
\end{itemize}
Our goal is to show under some assumptions on the input data that the algorithm leads to the construction of a simple $k$-hypergraph with positive probability. We start with some remarks on the strategy.

The key idea of our algorithm is to prevent the emergence of loops and thus only having to deal with parallel edges, providing good results also when $k$ is large. 
Indeed, in the second step we put the balls with the same label all into the same box, thus preventing the occurrence of loops. However, it is not clear how to always find such an allocation of the balls to the boxes. 

Note that it would of course be more intuitive to take only $k$ boxes instead of $k+1$, but this would mean that every box needs to be filled with exactly $\sigma/k$ balls, while still satisfying the constraints on putting all balls with the same label into the same box. This problem is called the \emph{multi-way number partitioning problem} and is known to be NP-hard~\cite{GJ}. By taking $k+1$ boxes instead, we have some margin on the fill heights that allows us to find such an allocation, under mild assumptions. One could also think about taking more than $k+1$ boxes. However, this does not improve the results but slightly weakens them.

Furthermore we remark that it is also not obvious that one can repeat the third step until all boxes are empty: it could be that we reach a stage where two boxes are empty, but there are still other non-empty boxes. 
It turns out that our assumptions on the degree sequence are sufficient to ensure that this will not happen.

\paragraph{Results.}
The following theorems state which assumptions guarantee that our general algorithmic approach of distributing balls into boxes will work. 
The pseudo-code of algorithms with the desired properties will be given in Section~\ref{sec:algos}.

\begin{Thm}\label{thm:uniform_distribution_of_half_edges}
	For $n\in\N$ and $k\geq 3$, let $\pi=(d_1,\dots,d_n)$ be a sequence such that \mbox{$k(k+1)d_{k+2}\leq \sigma$}. 
	Then there is a polynomial time randomized algorithm  
	that always returns a $k$-hypergraph $H$ with degree sequence $\pi$ and satisfies
	\begin{align*}
		\p(H\text{ is simple})\geq 1-\frac{k+1}{2}\bigg(\frac{3k}{2}\bigg)^{k-2} \frac{d_1^{k}}{\sigma^{k-2}}.
	\end{align*}
\end{Thm}

In the setting of the previous theorem one obviously has 
\begin{align}
        \p(H\text{ is simple})\to1 \label{eq:asympt_simple}
\end{align}
as $n\to\infty$ if
\begin{align}
	d_1^k=o(\sigma^{k-2}).\label{condition:2}
\end{align}
This is the case when the degrees in the sequence $\pi$ are 
either sufficiently small or sufficiently close to each other, as expressed in the following two corollaries.

\begin{Cor}
	Let $k\geq3$. If $d_1\leq Cn^\alpha$ for $C>0$ and $\alpha<1-\frac{2}{k}$, then~\eqref{eq:asympt_simple} holds.
\end{Cor}

\begin{Cor}
\label{cor:2}
	Let $k\geq3$ and define $\rho:=d_1/d_n$. 
	If 
	\[
	d_1^{2}\bigg(\frac{\rho}{n}\bigg)^{k-2} \to 0,
	\]
then~\eqref{eq:asympt_simple} holds.
\end{Cor}

Looking at the reduction for NP-hardness in~\cite{deza}, Corollary~2.1, it is clear that the decision problem remains NP-hard for $k\ge 4$ even when $d_1\le n^3$. On the other hand, using our Corollary~\ref{cor:2} it is now clear that for example for $k=15$ any sequence satisfying
$n^{5/2} \le d_n \le d_1 \le n^3$ is $k$-graphic if $n$ is sufficiently large.

Obviously, the applicability of Theorem~\ref{thm:uniform_distribution_of_half_edges} depends heavily on the role of $d_1$ in $\pi$. Consider for example the sequence 

\begin{equation}
	\label{eq:sequence}
	\pi:= \bigg(\underbrace{\frac{n}{\log(n)^3},\ldots,\frac{n}{\log(n)^3}}_{\log(n)}, \underbrace{\frac{\sqrt{n}}{\log(n)},\ldots,\frac{\sqrt{n}}{\log(n)}}_{n-\log(n)}\bigg),
\end{equation}

after appropriate roundings to obtain integers. Here we have 
$\sigma\approx \frac{n^{3/2}}{\log(n)}$ and hence Condition~\eqref{condition:2} is not satisfied for $k=4$. Therefore, we give another result which will allow us to ignore the first elements of the sequence $\pi$. 

\begin{Thm}\label{thm:many_small_degrees}
	For $k\ge 4$, $n\in\N$ and $\pi=(d_1,\dots,d_n)$, let $m\in[n]$ be maximal with
	\begin{align*}
		\frac{4\sigma}{k+1}\leq \sum_{i=m}^nd_i.
	\end{align*}
	If $k(k+1)d_{k-2}\leq \sigma$ and $5k(k+1)d_m\leq 4\sigma$,
	then there is a polynomial time randomized algorithm 
	that always returns a $k$-hypergraph $H$ with degree sequence $\pi$ and satisfies
	\begin{align*}
		\p(\text{H is simple})\geq 1-\frac{3k(k+1)}{4}\frac{d_m^{3}}{\sigma}.
	\end{align*}
\end{Thm}

Returning to our example sequence $\pi$ in~\eqref{eq:sequence}, we now have 
$d_m\approx\frac{\sqrt{n}}{\log(n)}$ and, again, $\sigma\approx \frac{n^{3/2}}{\log(n)}$, hence $d_m^3 = o(\sigma)$, thus proving that $\pi$ is indeed $4$-graphic for $n$ large enough.  

\paragraph{Related work.}
We briefly compare the above results to other activities in the area.  
Recently, Dyer et al.~\cite{Dyer} tried to generate simple hypergraphs with given degree sequence uniformly at random using a bijection between bipartite graphs and $k$-hypergraphs, which requires less assumptions than the configuration model for hypergraphs. Their methods allow for scenarios where $d_1=o(\min\{\sigma^{1/2},\sigma^{1-2/k}\})$ (see Theorem 1.6 in~\cite{Dyer}), while they require $d_1=O(\log n)$ for the configuration model (see Lemma~2.3 in~\cite{Dyer}). Our approach does not ask for uniform generation but works for scenarios up to $d_1 = o(\sigma^{1-2/k})$ (compare Condition~\eqref{condition:2}), thus improving the previous result for any $k\geq 5$ (the same result is obtained for $k=3,4$). Moreover, our Theorem~\ref{thm:many_small_degrees} allows us to ignore some vertices of higher degree.

Based on the characterization by Dewdney~\cite{dewdney}, in 2013 Behrens et al. gave sufficient conditions for a sequence to be $k$-graphic~\cite{behrens}. Among others they showed that a sequence is $k$-graphic if $d_1=o(\sigma^{1-1/k})$ (or even if ${d_{1}}/{\sigma^{1-1/k}}$ is less than some constant, Corollary~2.2 in~\cite{behrens}).  While this is a weaker constraint than our Condition~\eqref{condition:2}, their result is non-constructive whereas our methods allow us to generate a $k$-hypergraph with the given degree sequence in polynomial time.

\medskip
The remainder of this paper is organized as follows: in Section~\ref{sec:algos} we provide the implementation and analysis of our algorithms. Finally, in Section~\ref{sec:proofs} we formulate and prove a general result (Theorem~\ref{thm:main}) from which we then deduce Theorems~\ref{thm:uniform_distribution_of_half_edges} and~\ref{thm:many_small_degrees}.

In principle our methods should also apply to situations involving non-uniform hypergraphs, but then the statements and computations will be less appealing.

\section{Algorithms}\label{sec:algos}

We start with Step~2 of our approach sketched in the introduction, i.e.\ we need to distribute the balls representing half-edges among the $k+1$ boxes. To this end we use the following algorithm, employing a greedy strategy.

\begin{algorithm}[h!]\small
    
    \KwIn{$\pi=(d_{1},\dots,d_{n})$ a non-increasing sequence of natural numbers, $\ell\in\N$} 
    Set $B_1,\dots,B_{\ell}=\emptyset$\;

    \For{$i=1,\dots,n$}{
        
    Let $J\subseteq[\ell]$ be such that $|B_j|$ is minimal for all $j\in J$\;
    $j_{\min}=\min(J)$\;
    Add $d_i$ copies of the vertex $i$ to $B_{j_{\min}}$\;
    }
    \KwOut{$(B_1,\dots,B_\ell)$}
 
    \caption{\emph{greedy\_allocation($\pi,\ell$)}}
        \label{alg:greedy_all}
\end{algorithm}

Without further assumptions on the integer sequence $\pi$, it is not clear that one can use Algorithm \ref{alg:greedy_all} to fill the boxes $B_1,\dots,B_{k+1}$ in such a way that the demands in the second step of our approach are met: having all balls with the same label in a single box and no box exceeding $\sigma/k$ balls. The following definition provides a set-theoretic description of these requirements.

\begin{Def} 
	For $n\in\N$ and $\pi=(d_1\dots,d_n)$, we define $\mathrm{Allocation}_{k+1}(\pi)$ as the set of all $(k+1)$-tuples $(B_1,\dots,B_{k+1})$ of multi-sets $B_1,\dots,B_{k+1}$ such that $B_1,\dots,B_{k+1}$ are pairwise disjoint, each $i\in[n]$ is contained exactly $d_i$ times in one of the multi-sets and $\sigma/k\geq |B_1|\geq\ldots\geq |B_{k+1}|$.
\end{Def}

To ensure that the output of Algorithm \ref{alg:greedy_all} for $\ell=k+1$, after ordering by size, belongs to $\mathrm{Allocation}_{k+1}(\pi)$, we only need to control the cardinalities of the multi-sets $B_1,\dots,B_{k+1}$, as all other requirements are obviously satisfied. In Section \ref{sec:proofs}, we will check this condition via the following bound.

\begin{Lemma}\label{lem:allocation}
	For $\pi=(d_1,\dots,d_n)$ and $\ell\in\mathbb{N}$, the algorithm $\emph{greedy\_allocation}(\pi,\ell)$
	yields for all $i=1,\dots,\ell$,
	\begin{align*}
		|B_i|\leq\max \bigg(d_1,\frac{\sigma}{\ell}+d_{\ell+1}\bigg).
	\end{align*}
\end{Lemma}
\begin{proof}
	If we were able to distribute all vertices equally, we would obtain $\sigma/\ell$ objects in each box. Once some box $B_i$ contains more than $\sigma/\ell$ elements, there must be another box with fewer than $\sigma/\ell$ elements, so that we no longer put balls into $B_i$.
	
    In the first $\ell$ steps, we fill the vertices $1,\dots,\ell$ into the boxes $B_1,\dots,B_\ell$, respectively. Should $d_1$ exceed $\sigma/\ell$, we obtain $d_1$ as an upper bound on $B_1,\dots,B_\ell$ after the first $\ell$ steps. All boxes that are still below $\sigma/\ell$ can now overshoot $\sigma/\ell$ by at most $d_{\ell+1}$, as this is the largest degree that is left.
\end{proof} 

We continue with the algorithm for the third step of our approach, Algorithm~\ref{alg:sample_edges}, which samples the edges of the hypergraph.

\begin{algorithm}[h!]\small
    
    \KwIn{$\pi=(d_1,\dots,d_n)$ a non-increasing sequence of natural numbers and $(B_1,\dots,B_{k+1})\in\mathrm{Allocation}_{k+1}(\pi)$}

    Set $E=\emptyset$\;
    \For{$i=1,\dots,\sigma/k$}{

    Let $J\subseteq[k+1]$ be such that $|B_j|$ is minimal for all $j\in J$\;
    $j_{\min}=\min(J)$; 
    $\text{edge}=\emptyset$\;

    \For{$\ell=1,\dots,k+1,\ell\neq j_{\min}$}{
    choose $b_{\ell}$ uniformly at random from $B_{\ell}$\;
    $\text{edge}=\text{edge}\cup \{b_{\ell}\}$\;
    $B_{\ell}=B_{\ell}\backslash \{b_{\ell}\}$\;
 
    }
    
    $E=E\cup \{\text{edge}\}$; 
    }
    
    \KwOut{$E$} 
     \caption{\emph{sample\_edges($\pi,B_1,\dots,B_{k+1}$)}}
        \label{alg:sample_edges}
\end{algorithm}

 In Theorem~\ref{thm:main} we will gather some properties of \emph{sample\_edges} which are key ingredients for our proofs of Theorem \ref{thm:uniform_distribution_of_half_edges} and Theorem \ref{thm:many_small_degrees}. 

Finally, we present two algorithms that take a degree sequence and sample a $k$-hypergraph by combining Algorithm \ref{alg:greedy_all} and Algorithm \ref{alg:sample_edges}.

\begin{algorithm}[h!]\small 
    
    \KwIn{$\pi=(d_1,\dots,d_n)$ a non-increasing sequence of natural numbers, $k\geq3$} 
    $(B_1,\dots,B_{k+1})=\emph{greedy\_allocation($\pi,k+1$)}$\;
    Relabel $B_1,\dots,B_{k+1}$ such that they are decreasing in size\;
    \If{$(B_1,\dots,B_{k+1})\notin\mathrm{Allocation}_{k+1}(\pi)$}{
       \Return(Error)
    
    }
    $E$=\emph{sample\_edges($B_1,\dots,B_{k+1}$)}\;
    \KwOut{$E$}
 
    \caption{\emph{sample\_hypergraph($\pi,k$)}}
        \label{alg_sample_hypergraph}
\end{algorithm}

\begin{algorithm}[h!]\small 
    
    \KwIn{$\pi=(d_{1},\dots,d_{n})$ a non-increasing sequence of natural numbers, $k\geq 4,m\in[n]$} 
    $(B_1,\dots,B_{4})$=\emph{greedy\_allocation}($(d_m,\dots,d_{n}),4)$\;
    $(B_5,\dots,B_{k+1})$=\emph{greedy\_allocation}$((d_1,\dots,d_{m-1}),k-3)$\;
    Relabel $B_1,\dots,B_{k+1}$ such that they are decreasing in size\;
     \If{$(B_1,\dots,B_{k+1})\notin\mathrm{Allocation}_{k+1}(\pi)$}{
       \Return(Error)
    
    }
    $E$=\emph{sample\_edges($B_1,\dots,B_{k+1}$)}\;
    \KwOut{$E$}
 
    \caption{\emph{sample\_hypergraph\_2($\pi,k,m$)}}
        \label{algorithm4}
\end{algorithm}

Algorithm~\ref{alg_sample_hypergraph}, where we simply concatenate Algorithm \ref{alg:greedy_all} and Algorithm \ref{alg:sample_edges},  is more obvious, and is used to obtain Theorem \ref{thm:uniform_distribution_of_half_edges}. Algorithm~\ref{algorithm4} is designed for the situation of many vertices with small degree. It seems plausible that having many small degrees simplifies the task of avoiding parallel edges. An investigation of Algorithm~\ref{algorithm4} yields Theorem \ref{thm:many_small_degrees}.

\begin{Prop}\label{lem:runtime}
    Algorithm \ref{alg_sample_hypergraph} and Algorithm \ref{algorithm4} have computational costs of $O(kn+\sigma)$.
\end{Prop}
\begin{proof}
    It can be easily seen that Algorithm \ref{alg:greedy_all} has a  computational cost of $O(\ell n)$ while Algorithm \ref{alg:sample_edges} has a computational cost of $O(k\frac{\sigma}{k})=O(\sigma)$. Combining these observations concludes the proof.
\end{proof}
It is important to note that the computational cost depends on the choice of the parameter $k$, but this does not affect the  polynomiality of our strategy.

\section{Proofs}\label{sec:proofs}

The following theorem investigates the output of the algorithm \emph{sample\_edges} (see Algorithm \ref{alg:sample_edges}). Later on, we apply it to prove Theorem \ref{thm:uniform_distribution_of_half_edges} and Theorem \ref{thm:many_small_degrees}. For a multi-set $A$ we denote by $\mathrm{Supp}(A)=\{a\colon a\in A\}$ the underlying set, which no longer takes into account the multiplicities in $A$.
\begin{Thm} 
	\label{thm:main}
	Consider $n\in\N$ and let $\pi=(d_1,\dots,d_n)$. For $(B_1,\dots,B_{k+1})\in\mathrm{Allocation}_{k+1}(\pi)$ we have that
	\begin{enumerate}
		\item the algorithm \emph{sample\_edges}$(\pi,B_1,\dots,B_{k+1})$ terminates,
		\item provides a $k$-hypergraph $H$ without loops and with degree sequence $\pi$,
		\item and
	\end{enumerate}
	\begin{align}
	&\p(\text{H has no parallel edges})\nonumber\\
	&\geq 1-\sum_{\ell=1}^{k+1}\min_{j\in[k+1]\setminus\{\ell\}}\frac{|B_j|(|B_j|-1)}{2}\prod_{i\in[k+1]\setminus\{\ell\}}\frac{\max_{u\in \mathrm{Supp}(B_i)}d_u}{|B_i|}.\label{eq:thm_main_2nd_line}
	\end{align}
\end{Thm}

\begin{proof}
    We start by showing the first claim, i.e.\ that the algorithm terminates. From $(B_1,\dots,B_{k+1})\in\mathrm{Allocation}_{k+1}(\pi)$ we deduce  that $B_{k+1}$ is, in the beginning, among the boxes that contain the fewest elements, and it stays that way by construction (in case of a tie concerning the cardinalities $|B_1|,\dots,|B_{k+1}|$, \emph{sample\_edges} chooses the highest label, see Algorithm~\ref{alg:sample_edges}, line~4). We will show that, as soon as $B_{k+1}$ runs empty, all other boxes contain precisely one ball each.
    Since $\sigma/k\geq |B_1|\geq \hdots\geq |B_{k+1}|$, there exist $r_1,\dots,r_k\geq 0$ such that
    \begin{align*}
        |B_i|=\sigma/k-r_i \text{ for } i\in[k] \quad \text{ and }\quad|B_{k+1}|=\sum_{i=1}^kr_i.
    \end{align*}
    We compare $|B_{k+1}|$ to the number of balls missing to fill the first $k$ boxes to the height of $B_1$, the fullest one. This number is given by 
    \begin{align*}
        D=\sum_{i=2}^k|B_1|-|B_i|=\sum_{i=2}^k(r_i-r_1)\leq \sum_{i=1}^k r_i=|B_{k+1}|. 
    \end{align*}
    We update $D$ when drawing balls and investigate its changes. Whenever we draw a vertex from the last box, there are only two possible cases.
    \begin{itemize}
        \item \textit{We do not draw a vertex from the first box.} In this case, the first box must be among the boxes that contain the fewest vertices. But since the first box always contains the most vertices (again, by the choice made in case of a tie), we must have already reached $D=0$.
        
        \item \textit{We draw a vertex from the first box.} In this case, the discrepancy between the first box and the one we do not draw from gets reduced by one, whereas all the others stay the same. Therefore, $D$ gets reduced by one.
    \end{itemize}

    From the inequality $D\leq |B_{k+1}|$ above, we conclude that we reach a point where $D=0$, i.e.~the first $k$ boxes have the same number of elements in them before $B_{k+1}$ runs out of balls. When $D=0$, one keeps drawing balls from the first $k$ boxes until all $k+1$ boxes contain the same number of elements. From here on out, the difference between the number of balls in the fullest and least full box can be at most one. Since $k$ divides $\sigma$, there must be one ball in each of the first $k$ boxes when the last box runs empty, which shows the first claim.

   The second claim follows from $(B_1,\dots,B_{k+1})\in\mathrm{Allocation}_{k+1}(\pi)$, where the pairwise disjointedness of $B_1,\dots,B_{k+1}$ ensures the absence of loops.

    It remains to show the inequality in the third claim. For $\ell\in[k+1]$, let $E_\ell$ denote the list of all edges that do not contain a vertex from $B_\ell$. The order of the edges in the list $E_\ell$ shall be the order of their creation in the algorithm. Since $B_1,\dots,B_{k+1}$ are pairwise disjoint, it follows that two lists $E_i$ and $E_j$ cannot share an edge for $i\neq j$. Defining $A_\ell$ as the event that some edge occurs twice in $E_\ell$ for $\ell\in[k+1]$, we deduce that
    \begin{align}
        \p(H \text{ has no parallel edges})\geq1-\sum_{\ell=1}^{k+1}\p(A_\ell).\label{eq:prob_H_simple}
    \end{align}
    We proceed by giving an upper bound on $\p(A_\ell)$ for a fixed $\ell\in[k+1]$. It may be assumed that the boxes $B_i$ with $i\in[k+1]\setminus\{\ell\}$ contain at least two elements each, otherwise we would get $|E_\ell|\leq1$ and thus $\p(A_\ell)=0$. 
    Denote the elements in $E_\ell$ by $e_i$, with $i=1,\dots,|E_\ell|$, so that
    \begin{align}
        \p(A_\ell)\leq \sum_{1\leq i<j\leq |E_\ell|}\p(e_i=e_j). \label{eq:prob_A_ell}
    \end{align}
    To simplify notation, we write edges as vectors where we order the vertices according to the indices of $B_1,\dots,B_{k+1}$ they belong to. Additionally, we assume that the elements of $B_1,\dots,B_{k+1}$ are distinguishable even if they refer to the same vertex. Then, for distinct $i,j\in\{1,\dots,|E_\ell|\}$, the possible choices for $(e_i,e_j)$ are of the form $(f,g)$ given by
    \begin{align*}
        &((f_1,\dots,f_{\ell-1},f_{\ell+1},\dots,f_{k+1}),(g_1,\dots,g_{\ell-1},g_{\ell+1},\dots,g_{k+1}))\\
        &\in\bigg(\bigtimes_{i\in[k+1]\setminus\{\ell\}}B_i\bigg)^2,
    \end{align*}
    
    where $f_s\neq g_s$ for all $s\in[k+1]\setminus\{\ell\}$, because we think of the elements as distinguishable. Since all random choices are with respect to uniform distributions, each possible combination $(f,g)$ must have the same probability, so that
    \begin{align*}
        \p(e_i=f,e_j=g)=\prod_{s\in[k+1]\setminus\{\ell\}}\frac{1}{|B_s|(|B_s|-1)}.
    \end{align*}
    Now let us go back to indistinguishable objects in the boxes whenever they refer to the same vertex. Then we need to make up for the number of copies of a vertex, i.e.\ its degree, and obtain for a fixed possible edge $h\in\bigtimes_{i\in[k+1]\setminus\{\ell\}}\mathrm{Supp}(B_i)$ that
    \begin{align*}
        \p(e_i=e_j=h)=\prod_{s\in[k+1]\setminus\{\ell\}}\frac{d_{h_s}(d_{h_s}-1)}{|B_s|(|B_s|-1)}.
    \end{align*}
    Using the symmetry and summing over all possible choices for $h$, we obtain from \eqref{eq:prob_A_ell} that
    \begin{align*}
        \p(A_\ell)&\leq\frac{|E_\ell|(|E_\ell|-1)}{2}\p(e_1=e_2)\nonumber\\
        &=\frac{|E_\ell|(|E_\ell|-1)}{2}\prod_{i\in[k+1]\setminus\{\ell\}}\sum_{j\in \mathrm{Supp}(B_i)}\frac{d_j(d_j-1)}{|B_i|(|B_i|-1)}.
    \end{align*}
    Next we observe that all edges in $E_\ell$ need to contain vertices from all $B_i$, with $i\in[k+1]\setminus\{\ell\}$. We obtain
    \begin{align*}
        |E_\ell|\leq\min_{j\in[k+1]\setminus\{\ell\}}|B_j|.
    \end{align*}
    Moreover, it holds for all $i\in[k+1]$ that
    \begin{align*}
        \sum_{j\in \mathrm{Supp}(B_i)}d_j(d_j-1)&\leq \max_{u\in \mathrm{Supp}(B_i)}d_u\sum_{j\in \mathrm{Supp}(B_i)}(d_j-1)\\
        &\leq \max_{u\in \mathrm{Supp}(B_i)}d_u(|B_i|-1).
    \end{align*}
    Combining the three inequalities above with~\eqref{eq:prob_H_simple} yields~\eqref{eq:thm_main_2nd_line}.
\end{proof}

\begin{proof}[Proof of Theorem~\ref{thm:uniform_distribution_of_half_edges}]
   We consider Algorithm~\ref{alg_sample_hypergraph}, which has a polynomial runtime by Proposition \ref{lem:runtime}. Theorem~\ref{thm:main} immediately implies all other claims aside from the probability bound if we can show that   $(B_1,\dots,B_{k+1})\in\mathrm{Allocation}_{k+1}(\pi)$. The only property which is not clear by construction is that $B_1$ contains at most $\sigma/k$ elements (after relabeling the boxes in the second line of Algorithm~\ref{alg_sample_hypergraph}). From Lemma~\ref{lem:allocation} with $\ell=k+1$ it follows that 
    \begin{align*}
        |B_1|\leq \max\bigg(d_1,\frac{\sigma}{k+1}+d_{k+2}\bigg).
    \end{align*}
    By our assumption on the input sequence, we know that $d_1\leq \sigma/k$. On the other hand, using the assumed bound on $d_{k+2}$, we compute
    \begin{align*}
        \frac{\sigma}{k+1}+d_{k+2}\leq \frac{\sigma}{k+1}+\frac{\sigma}{k(k+1)}=\frac{\sigma}{k}.
    \end{align*}
    This implies $|B_1|\leq \sigma/k$ and thus $(B_1,\dots,B_{k+1})\in\mathrm{Allocation}_{k+1}(\pi)$.

    It remains to show the lower bound on the probability of $H$ being simple. By Theorem~\ref{thm:main}, the resulting $k$-hypergraph has no loops and the probability of having no parallel edges is bounded from below by
    \begin{align*}
        &\p(H\text{ has no parallel edges})\\
  &\geq 1-\sum_{\ell=1}^{k+1}\min_{j\in[k+1]\setminus\{\ell\}}\frac{|B_j|(|B_j|-1)}{2}\prod_{i\in[k+1]\setminus\{\ell\}}\frac{\max_{u\in \mathrm{Supp}(B_i)}d_u}{|B_i|}\\
  &\geq 1-\frac{k+1}{2}\frac{d_1^k}{|B_{k-1}|^{k-2}},
    \end{align*}
    where the second inequality follows from the inequalities $d_1\geq\hdots\geq d_n$ and $|B_1|\geq\hdots\geq |B_{k+1}|$. We obtain a lower bound on $|B_{k-1}|$ by observing that the first $k-2$ boxes all contain at most $\sigma/k$ elements each, so that there are at least $\sigma-(k-2)\sigma/k$ vertices left to distribute between $B_{k-1},B_k$ and $B_{k+1}$. Since $B_{k-1}$ contains the most elements among these three, it holds that
    \begin{align}
        |B_{k-1}|\geq \frac{1}{3}\bigg(\sigma-(k-2)\frac{\sigma}{k}\bigg)=\frac{2\sigma}{3k}.\label{eq:lower_bound_B_k-1}
    \end{align}
    Inserting this into the formula above yields
    \begin{align*}
        \p(H\text{ is simple})\geq 1-\frac{k+1}{2}\bigg(\frac{3k}{2}\bigg)^{k-2}\frac{d_1^k}{\sigma^{k-2}}
    \end{align*}
    and finishes the proof.
\end{proof}
\begin{proof}[Proof of Theorem~\ref{thm:many_small_degrees}]
We show that Algorithm~\ref{algorithm4} has the required properties. By Proposition~\ref{lem:runtime} it has a polynomial runtime. Due to Theorem~\ref{thm:main} it suffices to show that $(B_1,\dots,B_{k+1})\in\mathrm{Allocation}_{k+1}(\pi)$ in order to immediately obtain all remaining claims aside from the bound on the probability. The non-trivial condition to check is $|B_1|\leq \sigma/k$ (after relabelling in the third line of Algorithm~\ref{algorithm4}). Suppose that $B_1$ is generated in the first line of the code. Then Lemma~\ref{lem:allocation} implies that
    \begin{align*}
        |B_1|\leq \max\bigg(d_m,\frac{\sum_{i=m}^nd_i}{4}+d_{m+4}\bigg).
    \end{align*}

    By assumption we have $d_m\leq d_1\leq \sigma/k$. Since $m$ is maximal with the property
    \begin{align*}
        \sum_{i=m}^nd_i\geq \frac{4\sigma}{k+1},
    \end{align*}
    we deduce that
    \begin{align*}
        \quad\sum_{i=m+1}^nd_i< \frac{4\sigma}{k+1}.
    \end{align*}
    As $d_{m+4}\leq d_m$, the assumed bound on $d_m$ allows us to compute
    \begin{align*}
        \frac{\sum_{i=m}^nd_i}{4}+d_{m+4}<\frac{\sigma}{k+1}+\frac{5d_m}{4}\leq \frac{\sigma}{k+1}+\frac{\sigma}{k(k+1)}=\frac{\sigma}{k}. 
    \end{align*}
    Now suppose that $B_1$ is generated in the second line of the code of Algorithm~\ref{algorithm4}. Then Lemma~\ref{lem:allocation} provides
    \begin{align*}
        |B_1|\leq \max\bigg(d_1,\frac{\sum_{i=1}^{m-1}d_i}{k-3}+d_{k-2}\bigg).
    \end{align*}
  We have $d_1\leq\sigma/k$ whereas the definition of $m$, and the assumed bound on $d_{k-2}$ give us
    \begin{align*}
        \frac{\sum_{i=1}^{m-1}d_i}{k-3}+d_{k-2}&=\frac{\sigma-\sum_{i=m}^nd_i}{k-3}+d_{k-2}\leq \frac{\sigma-\frac{4\sigma}{k+1}}{k-3}+\frac{\sigma}{k(k+1)}\\
        &=\frac{\sigma}{k+1}+\frac{\sigma}{k(k+1)}=\frac{\sigma}{k}.
    \end{align*}
    In both cases we obtain $|B_1|\leq\sigma/k$, which allows to apply Theorem~\ref{thm:main}. So, the algorithm terminates, and provides a $k$-hypergraph with the desired degree sequence. Moreover, the probability of the $k$-hypergraph being simple satisfies
    \begin{align*}
		&\p(H\text{ is simple})\nonumber\\
  &\geq 1-\sum_{\ell=1}^{k+1}\min_{j\in[k+1]\setminus\{\ell\}}\frac{|B_j|(|B_j|-1)}{2}\prod_{i\in[k+1]\setminus\{\ell\}}\frac{\max_{u\in \mathrm{Supp}(B_i)}d_u}{|B_i|}.
	\end{align*}
 Let $a,b,c,d\in[k+1]$ be the indices of the boxes that were filled with the copies of the vertices $m,\dots,n$. In the rightmost product sign above we omit one factor $\ell\in[k+1]$ so that we have at least three elements of $a,b,c,d$ left in that product, all of which are not equal to $\ell$. We denote these three elements by $x_{\ell},y_{\ell}$ and $z_{\ell}$. As $\max_{u\in\mathrm{Supp}(B_i)}d_u\leq |B_i|$ for all $i\in[k+1]$, we derive 
 \begin{align*}
     \p(H\text{ is simple})&\geq 1-\sum_{\ell=1}^{k+1}\min_{j\in[k+1]\setminus\{\ell\}}\frac{|B_j|(|B_j|-1)}{2}\frac{d_m^3}{|B_{x_{\ell}}||B_{y_{\ell}}||B_{z_{\ell}}|}.
 \end{align*}
 Since $x_\ell,y_\ell,z_\ell\neq \ell$, we deduce that |$B_{x_\ell}|,|B_{y_\ell}|$ and $|B_{z_\ell}|$ are all at least as large as $\min_{j\in[k+1]\setminus\{\ell\}}|B_j|$. As they are also pairwise distinct, their maximum is larger than or equal to $|B_{k-1}|$ as $|B_1|\geq \hdots\geq|B_{k+1}|$. We obtain
 \begin{align*}
     \p(H\text{ is simple})&\geq 1-\sum_{\ell=1}^{k+1}\frac{1}{2}\frac{d_m^3}{|B_{k-1}|}\geq1-\frac{3k(k+1)}{4}\frac{d_m^3}{\sigma},
 \end{align*}where we inserted $3k|B_{k-1}|\geq2\sigma $ from~\eqref{eq:lower_bound_B_k-1} above as lower bound for $|B_{k-1}|$ in the last inequality. Note that \eqref{eq:lower_bound_B_k-1} does not depend on how  the balls were allocated to the boxes and is also applicable here. This finishes the proof.
\end{proof}

\end{document}